\documentclass[10pt,b5paper,twoside, headrule]{amsart}

\usepackage{amsfonts, amsmath, amssymb,latexsym}
\usepackage{epsfig}
\usepackage[curve]{xy}
\usepackage{algorithmic}
\usepackage{algorithm}
\usepackage{enumerate}
\usepackage{framed}
\usepackage{hyperref}
\usepackage{mathtools}
\usepackage{lmodern}
\usepackage[T1]{fontenc}

\usepackage{chngcntr}

\footskip=10mm\parskip 0.2cm\addtocounter{page}{0}
\newtheorem{thm}{Theorem}[section]

\newtheorem{prop}[thm]{Proposition}

\theoremstyle{definition}
\newtheorem{defn}[thm]{Definition}

\theoremstyle{remark}
\newtheorem{rem}[thm]{Remark}
\numberwithin{equation}{section}

\newcommand{\Z}{{\mathbb Z}}

\title[Supercuspidal lifts]{Modular supercuspidal lifts of weight $2$}

\author{\sc Iv\'an Blanco-Chac\'on}
\address{Department of Physics and Mathematics\\
School of Science\\
University of Alcal\'a\\
Madrid\\
Spain}
\email{ivan.blancoc@uah.es}
\thanks{I. Blanco-Chac\'on is partially supported by MTM2016-79400-P, by PID2022-136944NB-I00 and by Research Council of Finland (project \#351271, PI Camilla Hollanti). L. Dieulefait is partially supported by PID2022-136944NB-I00. This work was done in part while I. Blanco-Chacón was a visiting professor at the ANTA group at the Department of Mathematics and Systems Analysis, Aalto University, Finland. The department and the Aalto Science Institute are gratefully acknowledged for their support.  }

\author{\sc Luis Dieulefait}
\address{Department of Mathematics and Computer Science\\
Faculty of Mathematics and Computer Science\\
University of Barcelona\\
Barcelona\\
Spain}
\email{ldieulefait@ub.edu}

\begin{document}
\renewcommand\baselinestretch{1.2}
\renewcommand{\arraystretch}{1}
\def\base{\baselineskip}
\font\tenhtxt=eufm10 scaled \magstep0 \font\tenBbb=msbm10 scaled
\magstep0 \font\tenrm=cmr10 scaled \magstep0 \font\tenbf=cmb10
scaled \magstep0


\def\evenhead{{\protect\centerline{\textsl{\large{I. Blanco}}}\hfill}}

\def\oddhead{{\protect\centerline{\textsl{\large{On the non vanishing of the cyclotomic $p$-adic $L$-functions}}}\hfill}}

\pagestyle{myheadings} \markboth{\evenhead}{\oddhead}

\thispagestyle{empty}

\maketitle

\begin{abstract}Let $F/\mathbb{Q}$ be any totally real number field and $\frak{N}$ an ideal of its ring of integers of norm $N$ and define, for every even $n$, the $[F:\mathbb{Q}]$-dimensional multiweight $\textbf{n}=(n,...,n)$. We prove that for a non CM Hilbert cuspidal Hecke eigenform for $F$, say $f\in S_{\textbf{k}}(\Gamma_0(\frak{N}))$ with $k>2$ even, and a prime $p>\max\{k+1,6\}$ totally split in $F$ such that $p\nmid N$ and such that the residual mod $p$ representation $\overline{\rho}_f$ satisfies that $\mathrm{SL}_2(\mathbb{F}_p)\subseteq \mathrm{Im}(\overline{\rho}_f)$, there exists a  lift $\rho_g$ associated to a Hilbert modular cuspform for $F$, say $g\in S_{\textbf{2}}(\frak{N}p^2,\epsilon)$ for some Nebentypus character $\epsilon$ which is supercuspidal at each prime of $F$ over $p$. We also observe that our techniques provide an alternative proof to the corresponding statement for classical Hecke cuspforms  already proved by Khare \cite{khare} with classical techniques. Finally, we take the opportunity to include a corrigenda for \cite{dieulefait} which follows from our main result, which provides a congruence that puts the micro good dihedral prime in the level.
\end{abstract}

\section{Introduction}

The theory of congruences between modular forms dates back to the seventies with the seminal works \cite{serrecongruences} and \cite{sdcongruences}. Together with the case where the congruences involve changes in the weights, several important results have addressed the problem of adding or removing primes to/from the level: the so-called level lowering/raising results, the most notorious ones being those of Ribet (\cite{ribet1} and \cite{ribet2}), the first one providing a proof of Serre's epsilon conjecture, one of the key inputs to the proof of Fermat's Last Theorem. However, both results are restricted to the case of Steinberg primes.

With the change of century, the approach to establish this kind of congruence results switched to the construction of deformations of Galois representations with prescribed local types, together with the application of modularity lifting theorems \`a la Wiles. This approach was pioneered by Gee in \cite{geeautomorphic}. More recently, some results of weight change have been obtained by these techniques in our previous work \cite{blancodieulefait}.

The present work subscribes to this methodology. In particular, given a totally real field $F$ and an ideal $\frak{N}$ of its ring of integers of norm $N$, we start with a Hilbert modular cuspform for $F$, say $f\in S_{\textbf{k}}(\frak{N})$ with $\textbf{k}=(k,k,..,k)$ a parallel weight such that $k>2$ is even and with trivial character. Assume that $f$ is not a theta series, i.e., it has no CM. Denote by $K_f$ the field obtained by adjoining to $\mathbb{Q}$ the Hecke eigenvaloues of $f$ and by $\mathcal{O}_f$ its ring of integers. Take a prime $p>\max\{k+1,6\}$ with $p\nmid N$ totally split in $F$, choose a prime ideal $\frak{P}$ of $\mathcal{O}_f$ over $p$ and consider the representation attached to $f$ by Carayol, Blasius, Rogawski and others: $\rho_{f,\frak{P}}:G_F\to\mathrm{GL}_2(\mathcal{O}_{f,\frak{P}})$.

By \cite[Prop. 5.2]{dimitrov}, for all but finitely many primes $p$, there exists $\frak{P}$ in $\mathcal{O}_{f}$ such that $\mathrm{SL}_2(\mathbb{F}_q)\subseteq \mathrm{Im}(\overline{\rho}_{f,\frak{P}})$ for some power $q$ of $p$, so we can take $p$ as above so that $\mathrm{Im}(\overline{\rho}_{f,\frak{P}})$ is large. The main result of the present work is:

\begin{thm}Let $F$ be a totally real number field, $k>2$ even and $p>\max\{k+1,6\}$ totally split in $F$. Let $f\in S_{\textbf{k}}(\frak{N})$ be such that $p\nmid N$ as above. Suppose that $\mathrm{SL}_2(\mathbb{F}_p)\subseteq\mathrm{Im}(\overline{\rho}_{f,\frak{P}})$. Then, there exists a modular lift of $\overline{\rho}_{f,\frak{P}}$, say $\rho_{g,\frak{P}'}:G_F\to \mathrm{GL}_2(\mathcal{O}_{g,\frak{P}'})$, with $\frak{P}'$ a prime of $\mathcal{O}_g$ over $p$, such that $g\in S_{\textbf{2}}(\frak{N}p^2,\epsilon)$ for some Nebentypus character $\epsilon$ and such that $\rho_{g,\frak{P}'}$ is $\frak{p}$-supercuspidal for every prime $\frak{p}$ in $F$ over $p$.
\label{thm1}
\end{thm}

In the particular case in which $F=\mathbb{Q}$, starting with $h\in S_k(\Gamma_0(N))$, our result provides an alternative proof to the following statement, which was already proved in \cite[Thm. 6]{khare} by using more classical tools:

\begin{thm}Let $h\in S_k(\Gamma_0(N))$ be a newform  without CM and let $p$ be a prime such that $p>\max\{k+1,6\}$ and $p\nmid N$. Suppose that a residual mod $p$  representation $\overline{\rho}_{h,\frak{P}}$ associated to $h$ satisfies $\mathrm{SL}_2(\mathbb{F}_p)\subseteq \mathrm{Im}(\overline{\rho}_{h,\frak{P}})$ . Then, there exists a modular lift of $\overline{\rho}_{h,\frak{P}}$, say $\rho_{g,\frak{P}'}:G_{\mathbb{Q}}\to \mathrm{GL}_2(\mathcal{O}_{g,\frak{P}'})$, with $\frak{P}'$ a prime over $p$ in $\mathcal{O}_g$, such that $g\in S_2(Np^2,\epsilon)$ for some Nebentypus character $\epsilon$ and such that $\rho_{g,\frak{P}'}$ is $p$-supercuspidal. In particular, there exist infinitely many rational primes $p$ such that $\overline{\rho}_{h,\frak{P}}$ admits a $p$-supercuspidal modular lift of  weight $2$.
\label{thm2}
\end{thm}

In our proof we will specify the order of the character corresponding to the local type of the supercuspidal representation (see Proposition \ref{firstlift}), with a view to use it in the precise problem we will describe below. 

Our present work as well as \cite{blancodieulefait} introduce a technique which is useful in the resolution of some cases of Langlands functoriality, as in \cite{dieulefaitannals}, where the second author introduces the so-called safe chains of modular Galois representations. The concept of safe chain is context dependent: once a particular case of functoriality is considered, what is required is that the
chain obtained after applying to the given chain the group theoretical operation being considered has the property that
each congruence in it is amenable to the application of a suitable automorphy lifting theorem.

Another such chain is constructed in \cite{dieulefait}, where one of the main tools introduced is the concept and properties of so-called micro good dihedral primes, which are basically small primes that are added to the level through suitable congruences such that the resulting modular form is supercuspidal at these primes.

We will apply  Thm. \ref{thm2} to correct an inaccuracy that has been observed in the procedure used in \cite{dieulefait} to introduce such a micro good dihedral prime.  In the last section of this paper we explain how Thm. \ref{thm2} can be used to give a single congruence that allows to introduce such a micro good dihedral prime, thus serving as a corrigenda for \cite{dieulefait}.

We recall that the two main results of \cite{dieulefait} are on one hand, the automorphy of the $\mathrm{Sym}^5$ of level $1$ modular forms, a result that has been superseeded by several results of Clozel-Thorne (\cite{ct1}, \cite{ct2} and \cite{ct3}) as well as by Newton-Thorne (\cite{nt}), and on the other hand, the base change for arbitrary classical modular forms.

Furthermore, we mention that  the same chain that we are repairing here for \cite{dieulefait} is also used to deduce other cases of Langlands functoriality, for instance in \cite{GL2GL2GL2} and \cite{luissara} new cases of automorphy of tensor products of modular or automorphic representations are established. 

Our work is organised as follows:

In Section 2 we introduce the standard notations and facts required for our proof. In particular, we recall the concepts and basic facts of universal, frame-universal and versal deformation rings, fundamental characters of level 2, Hilbert-Samuel multiplicity and the precise meaning of $p$-supercuspidality. We do not provide the details and proofs of the claims we state but we refer instead to the relevant literature where such properties are proved in detail.

In Section 3, we recall several facts about inertial types, focusing in those which correspond to $p$-supercuspidal representations. In particular, we use a result by Savitt (Prop. \ref{savitt}) which allows us to prove Prop. \ref{firstlift} thus producing local supercuspidal lifts  which we integrate in a finite family of local lifts (archimedean, potentially semistable and inertia rigid) to construct a global deformation ring $\mathcal{R}_{global}$ for which we prove that it has positive dimension. For this proof, we follow a similar argument to that we introduced in \cite{blancodieulefait}. After that, appealing to a result of \cite{gee}, we conclude the existence of a point $\mathcal{R}_{global}\to\bar{\Z}_p\to 0$ which corresponds with a modular lift of parallel weight $2$ of $\overline{\rho_{f}}$ which is $\frak{p}$-supercuspidal at each prime $\frak{p}$ of $F$ over $p$, concluding hence the proof of Thm. \ref{thm1} and of Thm. \ref{thm2}.

Finally, in Section 4 we use Thm. \ref{thm2} and Prop. \ref{firstlift} to correct a mistake in \cite{dieulefait}, where the micro good dihedral prime $p=43$ is used to prove the automorphy of $\mathrm{Sym}^5$ of level $1$ modular forms.

\section{notations and preliminary facts}

As usual, for any $p$-adic field or number field, we write $\mathrm{G}_F$ for the absolute Galois group $\mathrm{Gal}(\overline{F}/F)$ and if $F$ is a number field and $S$ a finite set of places of $F$ (possibly containing the archimedean ones), $\mathrm{G}_{F,S}$ stands for $\mathrm{Gal}(F_S/F)$ where $F_S$ is the maximal extension of $F$ unramified outside $S$. For a rational prime $p>6$,  we set $G_p:=\mathrm{Gal}(\overline{\mathbb{Q}}_p/\mathbb{Q}_p)$ and we denote by $\omega:\mathrm{G}_{\mathbb{Q}}\to\mathbb{Z}_p^*$ the $p$-adic cyclotomic character. Let $k$ be a finite field of characteristic $p$.

Let $\mathbb{Q}_p^{ur}$ stand for the maximal unramified extension of $\mathbb{Q}_p$ and $I_p=\mathrm{Gal}(\overline{\mathbb{Q}}_p/\mathbb{Q}_p^{ur})$ for the inertia group. Now, let $\mathbb{Q}_{p,t}$ be the maximal tamely ramified extension of $\mathbb{Q}_p$ and denote $I_t=\mathrm{Gal}(\mathbb{Q}_{p,t}/\mathbb{Q}_p^{ur})$ the tame inertia group. The common residue field of $\overline{\mathbb{Q}}_p$, $\mathbb{Q}_p^{ur}$ and $\mathbb{Q}_{t,p}$ is $\overline{\mathbb{F}}_p$. There is an identification $\theta: I_t\cong\varprojlim\mathbb{F}_{p^n}^*$ and  hence a projection $\omega_2:I_t\to\mathbb{F}_{p^2}^*$. This projection factors through $\mathrm{Gal}(\mathbb{Q}_{p,t}/\mathbb{Q}_{p^2})$, where $\mathbb{Q}_{p^2}$ is the quadratic unramified extension of $\mathbb{Q}_p$. The fundamental characters of level $2$ are then $\omega_2$ and $\omega_2^p$.

We will still denote by $\omega$ the reduction of $\omega$ modulo $p$, unless a possibility of confusion may occur. Likewise, we will denote by $\tilde{\omega}_2$ the Teichmuller lift of $\omega_2$.

Let $\overline{\rho}:G_p\to \mathrm{GL}_2(k)$ be a residual representation with residual field $k$ and denote by $R(\overline{\rho})$ the universal deformation ring of $\overline{\rho}$ (if it exists) and by $R(\overline{\rho})^{\square}$ the universal frame deformation ring of $\overline{\rho}$ (which always exists, as well known). We say that $\overline{\rho}$ has trivial endomorphisms if $\mathrm{End}_{\mathrm{Gal}}(\overline{\rho})\cong k$. It is also well known after Ramakrishna, that if $\overline{\rho}$ has trivial endomorphisms then there exists the universal deformation ring $R(\overline{\rho})$.

Further, we will denote by $R^{ver}(\overline{\rho})$ the versal deformation ring for $\overline{\rho}$ introduced in \cite{hutan} Section 3.1, which is unique up to non-canonical isomorphism. By Lemma 2.1 of \cite{iyengar}, $R^{\square}(\overline{\rho})$ is formally smooth over $R^{ver}(\overline{\rho})$.

For a $p$-adic ring $\mathcal{O}$, the term \emph{complete Noetherian local $\mathcal{O}$-algebra} will be abridged by CNL-$\mathcal{O}$-algebra. If $R$ is a CNL-$\mathcal{O}$-algebra of of residual field $\mathbb{F}$, the ring $\overline{R}:=R\otimes\mathbb{F}$ is a local $\mathbb{F}$-algebra. Let $\frak{m}$ be its unique maximal ideal and $d$ its Krull dimension. We will denote by $e(R)$ the Hilbert-Samuel multiplicity of $\overline{R}$, namely, $d!$ times the leading coefficient of the polynomial in $n$ giving $\mathrm{dim}_F\left(\frac{\overline{R}}{\frak{m}^{n+1}}\right)$ for $n\gg 0$. For instance, $e(\mathcal{O}[[X]])=1$.

Let $F/\mathbb{Q}$ be a totally real finite extension and $\frak{N}$ an ideal of its ring of integers. Given a Hilbert cuspidal Hecke eigenform $f\in S_{\textbf{k}}(\Gamma_0(\frak{N}))$ for $F$ of parallel weight $\textbf{k}$ with even $k>2$  (analogously a classical Hecke eigenform $h\in S_k(\Gamma_0(N))$, we will denote by $K_f$ (analogously by $K_h$) its field of coefficients and by $\mathcal{O}_f$ (analogously $\mathcal{O}_h$) the corresponding rings of integers. 

If  $\frak{P}$ is a prime of $\mathcal{O}_f$ over $p$, denote by $\rho_{f,\frak{P}}:\mathrm{G}_{F}\to\mathrm{GL}_2(\mathcal{O}_{f,\frak{P}})$ the $p$-adic representation associated to $f$ by Blasius, Carayol, Rogawski, Wiles, Taylor and others.

For an automorphic form $\Theta$ over $F$ corresponding to a classical modular cusp form if $K=\mathbb{Q}$ or to a Hilbert modular cusp form $f$ otherwise, let $\pi_{\Theta}:\mathrm{GL}_2(\mathbb{A}_{F})\to\mathbb{C}$ be the adelic cuspidal automorphic representation associated  to $\Theta$. Consider the local decomposition $\pi_{\Theta}=\otimes_{\nu} \pi_{\nu}$, where $\nu$ runs over the archimedean or non-archimedean places of $F$. Then, $\Theta$ (or $\pi_{\Theta}$) is said to be supercuspidal at $\nu$ or $\nu$-supercuspidal (and, by abuse of notation, that $f$ is $\nu$-supercuspidal) if $\pi_{\nu}$ is so, namely, if its Jacquet module is zero-dimensional or equivalently, if it is not equivalent neither to a principal series, nor to a twisted Steinberg representation.

\begin{rem} For any classical or Hilbert modular eigenform $h$, since $p$ splits in $F$, taking $\frak{p}$ in $\mathcal{O}_F$ over $p$, given the family of local representations $\{\rho_{h,\frak{l}}|_{G_{\frak{p}}}:G_{\mathbb{Q}_p}\to\mathrm{GL}_2(\mathcal{O}_{\frak{l}}):\frak{l}\mbox{ prime of }\mathcal{O}_h\}$, by compatibility with the local Langlands correspondence, its attached admissible representation $\pi_{h,\frak{p}}$  is $\frak{p}$-supercuspidal if $\rho_{h,\frak{P}}|_{G_{\frak{p}}}$ is potentially crystalline for $\frak{P}$ over $p$ in $\mathcal{O}_h$ and its associated Weil-Deligne representation $WD(\rho_{h,\frak{P}}|_{G_{\frak{p}}})$ is induced from a character $\chi:W(\overline{\mathbb{Q}}_p/\mathbb{Q}_{p^2})\to \mathbb{C}^*$, not invariant by inner conjugation by $\mathrm{Gal}(\mathbb{Q}_{p^2}/\mathbb{Q}_p)$ (see \cite[Section 2]{dpt}). Abusing language for the sake of simplicity we will also say that $\rho_{h,\frak{P}}$ is $\frak{p}$-supercuspidal.
\end{rem}

Let $p>\max\{k+1,6\}$ be a prime such that $p\nmid N$ and suppose that $p$ splits in $F$. Let $\overline{\rho}_{f,\frak{P}}:G_{F}\to\mathrm{GL}_2(\mathbb{F}_{p^n})$ denote the reduction of $\rho_{f,\frak{P}}$ modulo  $\frak{P}$. For a classical modular form $h$, we have the analogous definition for $\overline{\rho}_{h,\frak{P}}:G_{\mathbb{Q}}\to\mathrm{GL}_2(\mathbb{F}_{p^n})$. Due to our choice of $p$, the Serre weight of $\overline{\rho}_{f,\frak{P}}$  (locally at any prime above $p$) is precisely $k$, since we are assuming in particular $p>k+1$, hence in the Fontaine-Laffaille range (since $p$ is totally split, we consider the Serre weight  defined exactly as in the case of classical modular forms).

Finally, since our goal is to prove Thm. \ref{thm1}, we will assume that the residual representations (for classical and Hilbert modular forms) have large image, namely, that $\mathrm{SL}_2(\mathbb{F}_p)$ is contained in the image of the residual representation.

\section{Proof of Theorem \ref{thm1}}

From now on, $F$ will be a totally real number field and $\frak{N}$ an ideal of its ring of integers of norm $N$. Let $f\in S_{\textbf{k}}(\frak{N})$ be a Hilbert cuspidal Hecke eigenform for $F$ with $\textbf{k}=(k,k,..,k)$ its parallel weight ($k>2$ even) and with trivial character. We will assume that $f$ is not a theta series, i.e., it has no CM. Denote by $K_f$ the field obtained by adjoining to $\mathbb{Q}$ the Hecke eigenvaloues of $f$ and by $\mathcal{O}_f$ its ring of integers. Take a prime $p>\max\{k+1,6\}$ with $p\nmid N$ totally split in $F$, choose a prime ideal $\frak{P}$ of $\mathcal{O}_f$ over $p$ and let $\rho_{f,\frak{P}}:G_F\to\mathrm{GL}_2(\mathcal{O}_{f,\frak{P}})$ be the global $\frak{P}$-adic Galois representation attached to $f$.

By \cite[Prop. 5.2]{dimitrov}, for all but finitely many primes $p$, there exists $\frak{P}$ in $\mathcal{O}_{f}$ such that $\mathrm{SL}_2(\mathbb{F}_q)\subseteq \mathrm{Im}(\overline{\rho}_{f,\frak{P}})$ for some power $q$ of $p$, so we can take $p$ as above so that $\mathrm{Im}(\rho_{f,\frak{p}})$ is large. Our goal is to prove:

\begin{thm}Let $F$ be a totally real number field, $p>\max\{k+1,6\}$ totally split in $F$ and $f\in S_{\textbf{k}}(\frak{N})$ as above. Suppose that the residual mod $\frak{P}$  representation $\overline{\rho}_{f,\frak{P}}$ satisfies $\mathrm{SL}_2(\mathbb{F}_p)\subseteq\mathrm{Im}(\overline{\rho}_{f,\frak{P}})$. Then, there exists a modular lift of $\overline{\rho}_{f,\frak{P}}$, say $\rho_{g,\frak{P}'}:G_F\to \mathrm{GL}_2(\mathcal{O}_{g,\frak{P}'})$, with $\frak{P}'$ a prime of $\mathcal{O}_g$ over $p$, such that $g\in S_{\textbf{2}}(\frak{N}p^2,\epsilon)$ for some Nebentypus character $\epsilon$ such that $\rho_{g,\frak{P}'}$ is supercuspidal at every prime of $F$ over $p$. 
\end{thm}

Recall that given a profinite group $\Pi$, a closed subgroup $I$, a $p$-adic ring $\mathcal{O}$ and a CNL $\mathcal{O}$-algebra $R$, a representation $\rho:\Pi\to\mathrm{GL}_2(R)$ is called $I$-ordinary if, denoting $M:=R\times R$ with the $\Pi$-module structure given by $\rho$, the $R$-submodule $M^I$ invariant by $I$ is free of rank $1$ over $R$ and a direct summand of $M$. In our setting, we consider $\Pi=G_{\frak{p}}$, $R=\mathcal{O}_{f,\frak{P}}$ and $I=I_{\frak{p}}$ where $\frak{P}$ (resp. $\frak{p}$) are primes of $\mathcal{O}_f$ (resp. of $F$) over $p$.

We will make use of the following result:
\begin{thm}Let $f\in S_{\textbf{k}}(\frak{N})$ be a Hilbert cuspidal Hecke eigenform of parallel weight $\textbf{k}=(k,k,..,k)$, with $k>2$ even, trivial character and such that $p$ is totally split in $F$ and $p\nmid N$. Let $\frak{P}$ be a prime of $\mathcal{O}_f$ over $p$ and let $\frak{p}$ be a prime of $F$ over $p$. Then:
\begin{itemize}
\item[a)] If $\rho_{f,\frak{P}}|_{G_{\frak{p}}}$ is ordinary, then 
$$
\overline{\rho}_{f,\frak{P}}|_{I_{\frak{p}}}\cong \left(\begin{array}{cc} \omega^{k-1} & *\\ 0 & 1\end{array}\right).
$$
\item[b)] If $\rho_{f,\frak{P}}|_{G_{\frak{p}}}$ is not ordinary then
$$
\overline{\rho}_{f,\frak{P}}|_{I_{\frak{p}}}\cong\left(\begin{array}{cc}\omega_2^{k-1} & 0\\ 0 & \omega_2^{p(k-1)}\end{array}\right).
$$
\end{itemize}
\label{arranque}
\end{thm}
\begin{proof}
First, since $p\nmid N$, it is well known that the local representation $\rho_{f,\frak{P}}|_{G_{\frak{p}}}$ is crystalline and has Hodge-Tate weights $\{0, k-1 \}$. 

Second, recall that $p$ is totally split in $F$ so that $G_{\frak{p}}\cong G_p$ for each $\frak{p}$ prime of $F$ over $p$, and, third, $p$ is in the Fontaine-Laffaille range since $p>k+1$ by assumption. 

Hence, as proved in \cite{bm} the result follows.
\end{proof}

Since our result also applies to classical modular cuspforms, we recall the analogous notions for the simpler classical setting. First, recall that a normalised cuspidal Hecke eigenform is $p$-ordinary if its $p$-th Fourier coefficient is a $p$-adic unit. We will also make use of the following result:
\begin{thm}Let $h\in S_k(\Gamma_0(N))$ be a normalised Hecke eigenform, $p$ a prime such that $p\nmid N$ and $\frak{P}$ a prime of $\mathcal{O}_h$ above $p$. Then:
\begin{itemize}
\item[a)] If $h$ is $p$-ordinary, then 
$$
\overline{\rho}_{h,\frak{P}}|_{I_{\frak{p}}}\cong \left(\begin{array}{cc} \omega^{k-1}& *\\ 0 & 1\end{array}\right).
$$
\item[b)] If $h$ is not $p$-ordinary then
$$
\overline{\rho}_{h,\frak{P}}|_{I_\frak{p}}\cong\left(\begin{array}{cc}\omega_2^{k-1} & 0\\ 0 & \omega_2^{(k-1)p}\end{array}\right).
$$
\end{itemize}
\label{arranqueclassic}
\end{thm}
\begin{proof}For a) this is well known (see, for instance the proof of \cite[Thm. 2.2]{joshikhare}). As for b) this is a result by Fontaine, a proof of which can be found in \cite[Section 6.8]{edixhoven}.
\end{proof}

In our proof of Thm. \ref{thm1} we will follow the guidelines we developed in \cite{blancodieulefait}. First, observe that for our input datum $f\in S_{\textbf{k}}(\Gamma_0(\frak{N}))$, it is $\mathrm{det}(\rho_{f,\frak{P}})=\omega^{k-1}$. Since our sought for $\rho_{g,\frak{P}'}$ must be a deformation of $\overline{\rho}_{f,\frak{P}}$ and $g$ must have parallel weight $2$, it must be $\mathrm{det}(\rho_{g,\frak{P}'})=\epsilon\omega$ for a finite order character $\epsilon$. Since on the other hand $\epsilon\omega$ should be a lift of $\overline{\omega}^{k-1}$, it must be $\overline{\epsilon}=\overline{\omega}^{k-2}$.

Our first task is to construct a global deformation ring $\mathcal{R}_{global}$ parametrizing potentially crystalline $p$-supercuspidal deformations of $\overline{\rho}_{f, \frak{P}}$ of determinant $\mu=\epsilon\omega$ and to show that $\mathrm{dim}(\mathcal{R}_{global})\geq 1$. After that, we will prove that $\mathcal{R}_{global}$ is finitely generated as an $\mathcal{O}$-module and that there exist a point $\psi: \mathcal{R}_{global}\to\overline{\mathbb{Z}}_p$ corresponding to a $p$-supercuspidal weight $2$ modular cuspform.

\subsection{The $p$-local universal deformation ring}

Given the fundamental exact sequence $1\to \mathrm{I}_p\to \mathrm{G}_p\to\mathrm{Gal}(\mathbb{Q}_p^{ur}/\mathbb{Q}_p)\to 1$, the local Weil group $W_p$ is defined as $W_p=\{\sigma\in \mathrm{G}_p: \sigma|_{\mathbb{Q}^{ur}_p}=\phi_p^r,\,r\in\mathbb{Z}\}$, where $\phi_p$ is a Frobenius element. It is a dense subgroup of $\mathrm{G}_p$ and clearly $\mathrm{I}_p\subseteq W_p$.

\begin{defn}(\cite[Def. 2.1.1.1]{bm})A Galois type of degree $2$ is a representation $\tau:I_p\to \mathrm{GL}_2(\overline{\mathbb{Q}}_p)$ of open kernel such that it extends to $W_p$.
\end{defn}
For a representation $\rho:G_p\to \mathrm{GL}_2(\overline{\mathbb{Q}}_p)$, denote by $WD(\rho)$ its associated Weil-Deligne representation (see \cite[Def. 2.4]{sa}). 
\begin{defn}Let $\tau$ be a Galois type of degree 2 and let $\rho:G_p\to\mathrm{GL}_2(\mathcal{O})$ be a deformation of $\overline{\rho}$ where $\mathcal{O}$ is a $p$-adic ring. We say that $\rho$ is of type $(k,\tau)$ where $k\geq 1$ if 
\begin{itemize}
\item $\rho$ is potentially crystalline and $\tau(\rho):=WD(\rho)|_{I_p}=(\tau,0)$.
\item $\rho$ has Hodge-Tate weights $\{0,k-1\}$, and
\item $\mathrm{det}(\rho)$ is a fixed lift of $\mathrm{det}(\overline{\rho})$ of the following form: the $(k-1)$-st power of the $p$-adic cyclotomic character times a finite character of order prime to $p$.
\end{itemize}
If $\overline{\rho}$ has trivial endomorphisms then, for $\iota: R(\overline{\rho})\to\mathcal{O}$ the morphism induced by $\rho$, we say that $\frak{p}:=\mathrm{Ker}(\iota)$ is also of type $(k,\tau)$.
\end{defn}

\begin{defn}[\cite{sa}] Let $\tau$ be a Galois type, $k\geq 1$ and $\overline{\rho}$ a residual representation as above. If it does not exist any prime $\frak{p}\in \mathrm{Spec}(R(\overline{\rho}))$ of type $(k,\tau)$, we define $R(k,\tau,\overline{\rho}):=\{0\}$, otherwise $\displaystyle R(k,\tau,\overline{\rho}):=R(\overline{\rho})/\cap\frak{p}$, where the intersection runs over all the primes $\frak{p}$ of type $(k,\tau)$.
\label{rtype}
\end{defn}

\begin{rem}As it is defined, it follows that $R(k,\tau,\overline{\rho})$ is the largest quotient of $R(\overline{\rho})$ such that for any $\rho$, potentially semistable deformation of $\overline{\rho}$, if 
\begin{itemize}
\item[a)] $\rho\otimes\mathbb{Q}_p$ is potentially crystalline of Hodge-Tate weight $\{0,k-1\}$ and
\item[b)] $WD(\rho\otimes\mathbb{Q}_p)|_{I_p}\cong\tau$,
\end{itemize}
then $\rho$ factors through $R(k,\tau,\overline{\rho})$.
\end{rem}

Let $\mu_{gal}(k,\overline{\tau},\overline{\rho}):=e(R(k,\tau,\overline{\rho}))$. In particular, notice that if $\mu_{gal}(k,\overline{\tau},\overline{\rho})\neq 0$ then $R(k,\tau,\overline{\rho})\neq \{0\}$.

Let $m\geq 1$ be a divisor of $p^2-1$ and $\tau=\tilde{\omega}_2^m\oplus\tilde{\omega}_2^{pm}$ a Galois type of order $\frac{p^2-1}{(m,p^2-1)}$. For $\tau$ to be a $p$-supercuspidal representation its order cannot divide $p-1$, hence, $m=p-1$ is admissible. Our starting point is the following result:

\begin{thm}[\cite{sa} Theorem 1.4] Let $p$ be any prime and $\overline{\rho}:G_p\to\mathrm{GL}_2(\overline{\mathbb{F}}_p)$ a residual representation with trivial endomorphisms and $\tau\cong\tilde{\omega}_2^m\oplus\tilde{\omega}_2^{pm}$ a Galois type with  $p+1\nmid m$. Write $m=i+(p+1)j$ with $i\in\{1,...,p\}$ and $j\in\mathbb{Z}/(p-1)\mathbb{Z}$. Then:
\begin{itemize}
\item[1)]If $\overline{\rho}_{I_p}\otimes\overline{\mathbb{F}}_p\in\left\{\left(\begin{array}{cc}
\omega^{i+j} & \lambda_1\\
0 & \omega^{1+j}
\end{array}\right),\left(\begin{array}{cc}
\omega^{1+j} & \lambda_2\\
0 & \omega^{i+j}
\end{array}\right), \right\}$ with $\lambda_1$ \textit{peu ramifi\'e} if $i=2$ and $\lambda_2$ \textit{peu ramifi\'e} if $i=p-1$, then $\mu_{gal}(2,\tau,\overline{\rho})=1$.
\item[2)]If $\overline{\rho}_{I_p}\otimes\overline{\mathbb{F}}_p\in\left\{ \omega_2^{p+m}\oplus\omega_2^{1+pm},\omega_2^{1+m}\oplus\omega_2^{p(1+m)}\right\}$ then $\mu_{gal}(2,\tau,\overline{\rho})=1$.
\end{itemize}
\label{savitt}
\end{thm}

For such a choice of $\tau$, we can apply Thms. \ref{savitt} and \ref{arranque} to $\overline{\rho}=\overline{\rho}_{f,\frak{P}}|_{G_{\frak{p}}}$:

\begin{prop}Let $f\in S_{\textbf{k}}(\frak{N})$ be a Hilbert cuspform which a normalised Hecke eigenform for $F$ with ring of coefficients $\mathcal{O}$, suppose $f$ is of parallel weight $\textbf{k}=(k,k,..,k)$, with $k> 2$ even and has trivial character.  Then, for each prime $p>\max\{k+1,6\}$ split on $F$ such that $p\nmid N$ and for every prime $\frak{p}|p$, there exists $m>0$, a prime $\frak{P}$ of $\mathcal{O}$ over $p$ and a potentially crystalline deformation $\rho:G_p\to\mathrm{GL}_2(\mathcal{O}_{\frak{P}})$ of $\overline{\rho}_{f,\frak{P}}|_{G_{\frak{p}}}$ of type $(2,\tau)$ where $\tau\cong \tilde{\omega}_2^m\oplus\tilde{\omega}_2^{pm}$ is a supercuspidal inertial type. Moreover, $m$ can be taken to be equal to $k+(p+1)(p-2)$.
\label{firstlift}
\end{prop}
\begin{proof}Consider $\overline{\rho}:=\overline{\rho}_{f,\frak{P}}|_{G_{\frak{p}}}$. Since $p>k+1$, we have that $p+1\nmid k$. Let $m:=k-1-p+t(p^2-1)$ with $t$ such that $m>0$, e.g. take $t\geq \lceil\frac{p+1-k}{p^2-1} \rceil$. Clearly $p+1\nmid m$. On the other hand, since $p$ is totally split in $F$ we have $I_{\frak{p}}\cong I_p$.

Case 1: if $\rho_{f,\frak{P}}|_{G_{\frak{p}}}$ is ordinary, writing $m=k+(p+1)(t(p-1)-1)$, since $k>1$, we apply case a) of Thm. \ref{arranque} having 
$$
\overline{\rho}|_{I_p}\cong\left(
\begin{array}{cc}
\omega^{k-1}& *\\
0 & 1
\end{array}
\right).
$$
Case 1.1: if $*\neq 0$ we are in case 1) of Thm. \ref{savitt}. Indeed: first of all, $\overline{\rho}$ clearly has trivial endomorphisms. Secondly, for our choice it is $j\equiv-1\pmod{p-1}$ hence $j+1\equiv 0\pmod {p-1}$, which is compatible with this case. Since $k>2$ we do not have to worry about the \textit{peu ramifi\'e} case in Thm. \ref{savitt}. Hence $\mu_{gal}(2,\tau,\overline{\rho})=1$ in this case.

Case 1.2: if $*=0$, then it's not true that $\overline{\rho}$ has trivial endomorphisms and in principle we cannot use \cite{sa}. The strategy here is as follows: first, observe that in this situation our residual representation $\overline{\rho}$ has the form $\chi_1\oplus\chi_2$ where $\chi_1\chi_2^{-1}\neq \omega^{\pm}$ for otherwise it would be either $k=2$, excluded by assumption, or $k\equiv 0,2\pmod {p-1}$, hence either $p\leq k-1$ or $p\leq k+1$, also contradicting our hypothesis. Now, we apply Remark 5.7 of \cite{hutan}:

Let $\overline{\rho_1}$ (resp. $\overline{\rho_2}$) be the unique non-split extension of $1$ by $\omega^{k-1}$ (resp. of $\omega^{k-1}$ by $1$), so that 
$$
\overline{\rho_1}|_{I_p}\cong\left(\begin{array}{cc}\omega^{k-1} & *\\0 & 1\end{array}\right)\mbox{ and }\overline{\rho_2}|_{I_p}\cong\left(\begin{array}{cc}1 & *\\0 & \omega^{k-1}\end{array}\right)\mbox{ with }*\neq 0.
$$
By Thm. \cite{sa} and the former Case 1.1 we know that $\mu_{gal}(2,\tau,\overline{\rho_1})=1$. For a residual representation $\psi:G_p\to\mathrm{GL}_2(\overline{\mathbb{F}}_p)$, in the same vein as in Def. \ref{rtype}, let us define $R^{ver}(2,\tau,\psi)$ to be $\{0\}$ if it does not exist any deformation of $\overline{\psi}$ of type $(k,\tau)$ and otherwise, $R^{ver}(2,\tau,\psi):=R^{ver}(\psi)/\cap\frak{p}$ where the intersection runs over all the primes $\frak{p}$ of type $(k,\tau)$, namely those primes of the form $\frak{p}=Ker(\iota)$ where $\iota: R(\overline{\psi})^{ver}\to R$, where $\iota$ is induced by a deformation $\rho$ of $\overline{\psi}$ of type $(k,\tau)$ and $R$ is a $p$-adic ring.

In particular, since for $i=1,2$ there exists $R(2,\tau,\overline{\rho_i})$ and $R(\overline{\rho_i})\cong R^{ver}(\overline{\rho_i})$, it is also $R(2,\tau,\overline{\rho_i})\cong R^{ver}(2,\tau,\overline{\rho_i})$. Now Rem. 5.7 of \cite{hutan} gives
$$
e(R^{ver}(2,\tau,\overline{\rho}))\geq e(R^{ver}(2,\tau,\overline{\rho_1}))=e(R(2,\tau,\overline{\rho_1}))=\mu_{gal}(2,\tau,\overline{\rho_1})=1.
$$
Since $R^{\square}(\overline{\rho})$ is formally smooth over $R^{ver}(\overline{\rho})$, then $R^{\square}(2,\tau,\overline{\rho})$ is formally smooth over $R^{ver}(2,\tau,\overline{\rho})$ and we can grant the existence of a framed deformation of $\overline{\rho}$ satisfying conditions a) and b) of Remark 1.5 and hence factoring by $R(2,\tau,\overline{\rho})$. We cannot grant in this case that $\mu_{gal}(2,\tau,\overline{\rho})=1$ but the existence of this framed deformation of type $(2,\tau)$ is enough for the purpose of our proof.

Case 2: if $\rho_{f,\frak{P}}|_{G_{\frak{p}}}$ is not ordinary, by case 2) of Thm. \ref{arranque}, we have $\overline{\rho}|_{I_p}\cong \omega_2^{k-1}\oplus \omega_2^{p(k-1)}$. Since $p+1\nmid k$ and $m=k-1-p+t(p^2-1)>0$, we have that $p+m\equiv k-1\pmod{p^2-1}$ and $1+pm\equiv p(k-1)\pmod{p^2-1}$, hence $\overline{\rho}_{I_p}$ is as in case 2) of Thm. \ref{savitt}, hence also in this case we have $\mu_{gal}(2,\tau,\overline{\rho})=1$.

All told, in the three cases, there exists a deformation $\rho:G_p\to\mathrm{GL}_2(\mathcal{O}_{\frak{P}})$ of $\overline{\rho}$ of type $(2,\tau)$ with $\tau\cong \tilde{\omega}_2^m\oplus\tilde{\omega}_2^{pm}$ a supercuspidal inertial type and  such that $\tau(\rho)\cong \tau$, its Hodge-Tate weight is $\{0,1\}$ and $\mathrm{det}(\rho)=\omega^{k-1}\epsilon$, with $\epsilon$ of finite order and $\overline{\epsilon}=1$. 
\end{proof}

As a particular case of this, we have:

\begin{prop}Let $f\in S_k(\Gamma_0(N))$ be a normalised eigenform with ring of coefficients $\mathcal{O}$ with $k>2$ even. Then, for each prime $p>k+1$ such that $p\nmid N$, there exists $m>0$, a prime $\frak{P}$ of $\mathcal{O}$ over $p$ and a potentially crystalline deformation $\rho:G_p\to\mathrm{GL}_2(\mathcal{O}_{\frak{P}})$ of $\overline{\rho}_{f,\frak{P}}|_{G_p}$ of type $(2,\tau)$ where $\tau\cong \tilde{\omega}_2^m\oplus\tilde{\omega}_2^{pm}$ is a supercuspidal inertial type. Moreover, $m$ can be taken to be equal to $k+(p+1)(p-2)$. 
\label{firstliftclassical}
\end{prop}

From now on, we will restrict ourselves to the case of Hilbert modular forms and only at the end of the article we will remark that our proof also applies to the classical modular case, which we will use to prove the corrigenda mentioned in the introduction.

For the rest of this section, we fix a prime $p>\max\{k+1,6\}$ such that $p\nmid N$, a prime $\frak{P}$ of $\mathcal{O}_f$, a prime $\frak{p}$ of $F$ both over $p$ and a potentially crystalline $p$-supercuspidal lift $\rho$ of type $(2, \tau)$ of $\overline{\rho} := \overline{\rho}_{f,\frak{P}}|_{G_{\frak{p}}}$ as the one constructed in the previous proposition. For such a fixed prime $p$, let 
$R^{\square}(\overline{\rho})_{crys}$ be the quotient of $R^{\square}(\overline{\rho})$ corresponding to potentially crystalline framed deformations of $\overline{\rho}$ and let $R^{\square,\mu}(\overline{\rho})_{crys}\subseteq R^{\square}(\overline{\rho})_{crys}$ be the quotient of $R^{\square}(\overline{\rho})$ corresponding to potentially crystalline frame deformations of $\overline{\rho}$ of weight $2$ and determinant $\mu=\omega\epsilon$ equal to the determinant of $\rho$. 

\begin{prop}The quotient ring $R^{\square,\mu}(\overline{\rho})_{crys}$ is well defined and non-zero. Moreover $\mathrm{dim}_{\mathcal{O}}(R^{\square,\mu}(\overline{\rho})_{crys})=4$, where dim denotes the Krull dimension.
\label{krull}
\end{prop}
\begin{proof}  

The ring is non-zero since it contains in particular the lift $\rho$ just described. As for its dimension, Corollary 3.3.3 of \cite{allen} ensures that $R^{\square,\mu}(\overline{\rho})_{crys}$ is $\mathcal{O}$-flat, reduced, and that $\mathrm{dim}(R^{\square,\mu}(\overline{\rho})_{crys}[1/p])=4$. 

Now, reasoning as in \cite{eg} Lemma 4.3.1 we see that $R^{\square}(\overline{\rho})_{crys}$ is a power series ring in one variable over $R^{\square,\mu}(\overline{\rho})_{crys}$, thus $R^{\square,\mu}(\overline{\rho})_{crys}$ is $\mathcal{O}$-flat if and only if $R^{\square}(\overline{\rho})_{crys}$ is so, as is the case. Now, by the same argument as in the proof of \cite{kwannals} Prop. 2.12: 
$$
\mathrm{dim}_{\mathcal{O}}(R^{\square,\mu}(\overline{\rho})_{crys})=\mathrm{dim}(R^{\square,\mu}(\overline{\rho})_{crys}[1/p])=4.
$$
\end{proof}

\subsection{The global deformation ring $\mathcal{R}_{global}$}

Let $\frak{P}$ be a fixed prime of $\mathcal{O}_f$ over $p$ and denote by $\mathcal{O}$ the ring of integers of $K_{f,\frak{P}}$. This subsection follows the same strategy and steps as Section 3.2 of \cite{blancodieulefait}, the only difference being that at our prime $\frak{p}$ of $F$ over $p$, our local deformation ring corresponds to $\frak{p}$-supercuspidal deformations of Galois type $(2,\tau)$ instead of potentially diagonalizable ones. For that reason, we have omitted most of the details in some proofs, referring the reader to the corresponding sections therein. 

For each place $\nu$ of $F$ (finite or not) and for every $\frak{P}$ prime in $\mathcal{O}_f$ above $p$, let us denote by $\overline{\rho}_{\nu}$ the restriction $\overline{\rho}_{f,\frak{P}}|_{G_{\nu}}$.

\begin{defn}Let $S=\left\{\nu\mid \frak{N}, \nu\mbox{ prime in }F\right\}\cup\left\{\nu|p\mbox{ prime in }F\right\}\cup\{\nu|\infty\mbox{ archimedean prime}\}$ and for each $\nu\in S$, define the following ring $R^{\square,\mu}_{\mathcal{O},\nu}$:
\begin{itemize}
\item for $\nu=\frak{p_i}|p $,  define $R^{\square,\mu}_{\mathcal{O},\nu}$ 
 to be the irreducible component of $R^{\square,\mu}(\overline{\rho}_{\nu})_{crys}$ containing the $\frak{p}$-supercuspidal weight $2$ lift $\rho$ constructed using proposition \ref{firstlift}. We know from proposition \ref{krull} that this ring has Krull dimension $4$. 
 Due to \cite{ki3} Theorem 2.7.6 (see also \cite{gee}  page 25), the Galois type is constant in $R^{\square,\mu}_{\mathcal{O},\nu}$, so any lift corresponding to a $p$-adic point of this ring will be potentially crystalline, $\frak{p}$-supercuspidal and of type $(2, \tau)$.
\item for $\nu|\infty$ an archimedean prime, define $R^{\square,\mu}_{\mathcal{O},\nu}$ to be framed universal deformation ring for $\overline{\rho}_{\nu}$ corresponding to odd deformations of $\overline{\rho_{f,\frak{P}}}$ with determinant $\mu$. 
\item for $\nu\nmid p,\infty$, if $\overline{\rho}_{\nu}$ is not the twist of a semistable representation, define $R^{\square,\mu}_{\mathcal{O},\nu}$ to be the framed universal deformation ring corresponding to deformations $\rho$ such that $\rho(I_{\nu})$ is finite and $\mathrm{det}(\rho)=\mu$. 
\item finally, for $\nu\nmid p,\infty$,  if $\overline{\rho}_{\nu}$ is the twist of a semistable representation, define $R^{\square,\mu}_{\mathcal{O},\nu}$ to be the framed universal deformation ring corresponding to semistable lifts of determinant $\mu$.
\end{itemize}
\label{defnlocal}
\end{defn}

We know that $\mathrm{dim}_{\mathcal{O}}(R^{\square, \mu}_{\mathcal{O},\nu})=4$ for $\nu\mid p$, and we can put all the dimensions of the former rings in the following proposition, for a proof of the second and third statement of which we refer to \cite[Prop. 3.3]{blancodieulefait}:
\begin{prop}The rings introduced in \ref{defnlocal} are finite $\mathcal{O}$-dimensional and $\mathcal{O}$-flat. Their dimensions are:
\begin{itemize}
\item[1. ] $\mathrm{dim}_{\mathcal{O}}(R^{\square, \mu}_{\mathcal{O},\nu})=4$, for each $\nu\mid p$
\item[2. ] for $\nu\nmid p,\infty$, $\mathrm{dim}_{\mathcal{O}}(R^{\square, \mu}_{\mathcal{O},\nu})=3$,
\item[3. ] $\mathrm{dim}_{\mathcal{O}}(R^{\square, \mu}_{\mathcal{O},\nu})=2$ for each $\nu\mid\infty$.
\end{itemize}
\end{prop}

The next step is to \textit{paste together} the previous rings in the following sense:
\begin{prop}[\cite{blancodieulefait} Proposition 3.5] The ring $R^{\square, loc,\mu}_S:=\displaystyle\widehat{\otimes}_{\nu\in S}R_{\mathcal{O},\nu}^{\square,\mu}$ is flat over the ring of integers of a finite extension of $\mathbb{Q}_p$ and has relative dimension $3|S|$.
\label{dim1}
\end{prop}
\begin{proof}For primes $\nu\mid p$, the rings $R^{\square, \mu}_{\mathcal{O},\nu}$ contribute $4n$, where $n=[F:\mathbb{Q}]$. For primes $\nu|\infty$, the rings $R^{\square, \mu}_{\mathcal{O},\infty}$ contribute $2n$. Finally, for primes $\nu\nmid p,\infty$, the rings $R^{\square, \mu}_{\mathcal{O},\nu}$ contribute $3(|S|-2n)$, hence the result follows.
\end{proof}
Next, define 
\begin{equation}
\widehat{R}_S^{\square, loc}:=\widehat{\otimes}_{\nu\in S} R_{\nu}^{\square,\mu},
\end{equation}
where $R_{\nu}^{\square, \mu}$ denotes the framed deformation ring of $\overline{\rho}_{\nu}$ representing the functor assigning to a CNL $\mathcal{O}$-algebra A the isomorphism classes of lifts of $\overline{\rho}_{\nu}$ to $\mathrm{GL}_2(A)$ with determinant $\mu$. 

Furthermore, let now $R^{\square, \mu}_{\mathcal{O},S}$ be the universal ring representing the functor assigning a CNL $\mathcal{O}$-algebra $A$ the set of isomorphism classes of pairs $ (\rho_A,\{\beta_{A,\nu}\}_{\nu\in S})$ where $\rho_A$ is a deformation of $\overline{\rho}_{f,\frak{P}}$ to $\mathrm{GL}_2(A)$ with $\det(\rho_A)=\mu$ unramified outside $S$.  
 As proved in \cite{blancodieulefait} Proposition 3.6, the ring $R^{\square, \mu}_{\mathcal{O},S}$ is an $\widehat{R}^{\square, loc}_S$-algebra.
Finally, we present the global object we seek for:
\begin{defn}Denoting $\widehat{R}^{\square,\mu}_S:=R^{\square,\mu}_{\mathcal{O},S}\widehat{\otimes}_{\widehat{R}_S^{\square,loc}}R^{\square,loc,\mu}_S$, let us define $\mathcal{R}_{global}$ as the image of the unframed universal deformation $R_S^{\mu}$ in $\widehat{R}^{\square,\mu}_S$ induced by the canonical map $R_S\to R^{\square,\mu}_{\mathcal{O},S}$. Notice that $R^{\mu}_S$ is well defined since $\mathrm{SL}_2(\mathbb{F}_p)\subseteq\mathrm{Im}(\overline{\rho}_{f,\frak{p}})$ so that in particular this residual representation is absolutely irreducible.
\label{globalring}
\end{defn}

Since for every $\nu\nmid p$ the ring $R^{\square, \mu}_{\mathcal{O},\nu}$ has been defined as in \cite{kwannals}, the proof of the next result is the same as that of \cite[Prop. 3.9]{blancodieulefait}:
\begin{prop}The ring $\widehat{R}^{\square,\mu}_S$ is a power series ring over $\mathcal{R}_{global}$ in $4|S|-1$ variables.
\label{dim2}
\end{prop}

\subsection{The dimension of $\mathcal{R}_{global}$}
For a number field $F$ and a finite set $S$ of places of $F$ possibly containing the archimedean ones, and given a $G_F$-module $M$ with unramified Galois action outside $S$, let us set $H^k(S,M):=H^k(G_{F,S},M)$. The usual adjoint representations are denoted $Ad:=Ad(\overline{\rho}_{f,\frak{P}})$, $Ad^0:=Ad^0(\overline{\rho}_{f,\frak{P}})$, $(Ad^0)^*:=\mathrm{Hom}_{\mathbb{F}}(Ad^0,\mathbb{F})$ and $(Ad^0)^*(1)=\mathrm{Hom}_{\mathbb{F}}(Ad^0,\mu_p^*)$ where $\mu_p$ denotes the multiplicative group of $p$-roots of unity. If for each $\nu\in S$ we are given a subspace $L_{\nu}\subseteq H^k(D_{\nu},M)$, we denote by $H^k_{\{L_{\nu}\}}(S,M)$ the preimage of $\prod_{\nu\in S}L_{\nu}\subseteq \prod_{\nu\in S}H^k(D_{\nu},M)$ under the restriction map $H^k(S,M)\to\prod_{\nu\in S}H^k(D_{\nu},M)$. 

Following \cite{blancodieulefait} again, let us define:
\begin{itemize}
\item for $M=Ad^0$ and $\nu\in S$, $L_{\nu}$ is the image of $H^0(D_{\nu}, Ad/Ad^0)$ in $H^1(D_{\nu},Ad^0)$.
\item for $M=Ad$, and $\nu\in S$, $L_{\nu}=\{0\}$. 
\end{itemize}
Observe that since $p>6$, $\mathrm{dim}(L_{\nu})=0$ for $M=Ad^0$. Further, we can write $Ad=Ad^0\oplus Z$ where $Z$ denotes the subspace of scalar matrices in $M_2(\mathbb{F}_p)$. 

Now, considering the natural exact sequence
$$
0\to H^0(S,Ad^0)\to H^0(S,Ad)\to \mathbb{F}\to H^1(S,Ad^0)\to H^1(S,Ad),
$$
we denote by $H^1(S,Ad^0)^{\eta}$ and $H^1_{\{L_{\nu}\}}(S,Ad^0)^{\eta}$ the images of $H^1(S,Ad^0)\to H^1(S,Ad)$ and $H^1_{\{L_{\nu}\}}(S,Ad^0)\to H^1(S,Ad)$ respectively. Then:
\begin{itemize}
\item[1. ] The surjections $H^1(S,Ad^0)\to H^1(S,Ad^0)^{\eta}$ and $H^1_{\{L_{\nu}\}}(S,Ad^0)\to H^1_{\{L_{\nu}\}}(S,Ad^0)^{\eta}$ are isomorphisms.
\item[2. ] $H^0(F,Ad^0)=H^0(F, (Ad^0)^{*}(1))=0$.
\item[3. ] $H^1(F,Z)\hookrightarrow H^1(F,Ad)$.
\end{itemize}

The following result is proved in Proposition 3.11 of \cite{blancodieulefait}:

\begin{prop}The minimal number of generators of $R^{\square, \mu}_{\mathcal{O},S}$ (analogously $\widehat{R}^{\square, \mu}_S$) over $\widehat{R}_S^{\square,loc, \mu}$ (analogously $R^{\square,loc, \mu}_S$) is 
$$
g:=\mathrm{dim}_{\mathbb{F}}(H^1_{\{L_{\nu}\}}(S,Ad^0))+\sum_{\nu\in S}\mathrm{dim}_{\mathbb{F}}(H^0(D_{\nu},Ad))-\mathrm{dim}_{\mathbb{F}}(H^0(F,Ad)).
$$
\label{ngen}
\end{prop}
We can now prove the main result of this section:
\begin{thm}Notations as before, we have that $\mathrm{dim}(\mathcal{R}_{global})\geq 1$.
\label{cotadimension}
\end{thm}
\begin{proof}First, from Proposition \ref{dim2} we have that
$$
\mathrm{dim}(\widehat{R}^{\square,\mu}_S)=\mathrm{dim}(\mathcal{R}_{global})+4|S|-1.
$$
Henceforth, it is enough to prove that $\mathrm{dim}(\widehat{R}^{\square,\mu}_S)\geq 4|S|$. To that end, we make use of Wiles's Lemma (see \cite[Lem. 3.10 and 3.12]{blancodieulefait}) to write:
$$
g=\mathrm{dim}(H^1_{\{L_{\nu}^*\}}(S,(Ad^0)^*(1)))-\mathrm{dim}(H^0(F,Ad))+\sum_{\nu\in S}\mathrm{dim}(H^0(D_{\nu},Ad))-\mathrm{dim}(H^0(D_{\nu,Ad^0})).
$$
Second, applying this equality to the natural exact sequence
$$
0\to H^0(D_{\nu},Ad^0)\to H^0(D_{\nu},Ad)\to\mathbb{F}\to 0,
$$
we obtain
\begin{equation}
g=\mathrm{dim}(H^1_{\{L_{\nu}^*\}}(S,(Ad^0)^*(1)))+|S|-1.
\label{ng}
\end{equation}
Since $R^{\square, \mu}_{\mathcal{O},S}$ has a presentation
$$
R^{\square, \mu}_{\mathcal{O},S}\cong \widehat{R}_S^{\square,loc, \mu}[[X_1,...,X_g]]/J 
$$
for some ideal $J$, let us set $r(J)$ to be the minimal number of generators of $J$, so that $r(J)=\mathrm{dim}(J/\frak{m}J)$, with $\frak{m}$ the maximal ideal of $\widehat{R}_S^{\square,loc, \mu}[[X_1,...,X_g]]$. In \cite[Lem. 3.14]{blancodieulefait} we prove that
$$
r(J)\leq \mathrm{dim}(H^1_{\{L_{\nu}^*\}}(S,(Ad^0)^*(1))).
$$ 
From this bound, we obtain another presentation
$$
\widehat{R}_S^{\square, \mu}\cong  R_S^{\square,loc, \mu}[[X_1,...X_g]]/J',
$$
for some other ideal $J'$ generated by at most $\mathrm{dim}_{\mathbb{F}}(H^1_{\{L_{\nu}^*\}}(S,(Ad^0)^*(1)))$ elements. Hence
$$
\mathrm{dim}(\widehat{R}_S^{\square, \mu})\geq \mathrm{dim}(R_S^{\square,loc, \mu})+g-\mathrm{dim}_{\mathbb{F}}(H^1_{\{L_{\nu}^*\}}(S,(Ad^0)^*(1))), 
$$
which due to Equation \ref{ng} is lower bounded by $\mathrm{dim}(R_S^{\square,loc, \mu})+|S|-1$, which equals $4|S|$, by Prop. \ref{dim1}.
\end{proof}

\subsection{The existence of modular $p$-supercuspidal lifts}
We start by recalling the following result, which will be applied in our proof:
\begin{thm}[\cite{gee}, Theorem 4.4.1] Let $F$ be an imaginary CM field and $F^{+}$ its maximal totally real subfield. Let $n\geq 1$ be an integer and $p>2(n + 1)$ an odd prime such that  $\zeta_p\not\in F$ and all primes of $F^{+}$ above $p$ split in $F$.  Let $S$ be a finite set of archimedean places of $F^{+}$, including all places above $p$, such that each place in $S$ splits in $F$. For each $\nu\in S$ choose a place $\tilde{\nu}$ of $F$ lying over $\nu$. Let $\mu$ be an algebraic character of $G_{F^{+}}$ and $\overline{r}:G_F \to GL_n(\bar{\mathbb{F}}_p)$ a continuous representation such that

\begin{itemize}
\item[1.] $ (\overline{r}, \overline{\mu})$ is a polarized mod $p$ representation unramified outside $S$, either ordinarily automorphic or potentially diagonalizably automorphic,
\item[2.] $\bar{r}|G_{F(\zeta_p)}$ is irreducible.
\end{itemize}

For each $\nu\in S$, let $\rho_{\nu} : G_{F_{\tilde{\nu}}} \to GL_n(\mathcal{O}_{\mathbb{Q}_p})$ be a lift of $\overline{r}|_{G_{F_{\tilde{\nu}}}}$. If $\nu|l$, assume further that $\rho_{\nu}$ is potentially diagonalizable, and that for all $\tau: F_{\tilde{\nu}}\to \overline{\mathbb{Q}}_p$, $HT_{\tau}(\rho_{\nu})$ consists of $n$ distinct integers.

Then there is a regular algebraic, cuspidal, polarized automorphic representation $(\pi, \chi)$ of $GL_n(\mathbb{A}_F )$ such that
\begin{itemize}
\item[(1)] $\overline{r_{p,\iota}(\pi)}\cong \bar{r}$;
\item[(2)] $r_{p,\iota}(\chi)\varepsilon_l^{1-n} = \mu$;
\item[(3)] $\pi$ has level potentially prime to $p$;
\item[(4)] $\pi$ is unramified outside $S$;
\item[(5)] for each $\nu\in S$ we have that $\rho_{\nu} $  connects to $ r_{p,\iota}(\pi)|_{G_{F_{\nu}}}.$
\item[(6)]  Suppose  that for all $\nu \mid l$ the lifts $\rho_{\nu} $ are crystalline. Then for all such $\nu$ the representation $ r_{p,\iota}(\pi)|_{G_{F_{\nu}}}$ is crystalline.
\item[(7)] Define the ring $\mathcal{R}_{F}$ as the universal deformation ring of $\bar{r}$ for the deformation problem with local conditions at primes $\nu$ dividing $p$ corresponding to fixing the irreducible component of the corresponding local deformation ring (of potentially crystalline representations with fixed Hodge-Tate weights) that contains $\rho_{\nu} $, and similarly for primes $\nu \in S$ not dividing $p$. Then $\mathcal{R}_{F}$ is a finitely generated $\mathcal{O}$-module and it has at least one point in $\overline{\mathbb{Z}}_p$, corresponding to the representation $r_{p,\iota}(\pi)$.

\end{itemize}
\label{geetool}
\end{thm}

\begin{rem}Items (6) and (7) of Thm. \ref{geetool} are not listed in the conclusion of this theorem in \cite{gee}. However these facts are proved as part of the proof therein. For details, cf. \cite{blancodieulefait} Rem. 3.16.
\end{rem}

Our aim is now to apply Thm. \ref{geetool} to a base change of $\bar{\rho}_{f,\frak{P}}$ to an imaginary CM field. To that end let us choose $S$ as in Section 3.2, namely, the union of the set of non-archimedean places pcorresponding to prime divisors of $\frak{N}$, to prime divisors of $p$ and the archimedean places of $F$. Let us also choose an imaginary quadratic extension $L/F$ such that all the archimedean places in $S$ are split. 
\begin{defn}
Let $\mathcal{R}_{global, L}$ be the global deformation ring introduced in Defn. \ref{globalring} but replacing $F$ by $L$.
\label{globalringF}
\end{defn}
To apply Thm. \ref{geetool} to $\bar{\rho}_{f,\frak{P}}|_{G_L}$, let us consider local lifts $\rho_\nu$ for each $\nu\in S$ as in Defn. \ref{defnlocal}: at a prime $\frak{p}$ over $p$ we choose the $\frak{p}$-supercuspidal lift from Lemma \ref{firstlift}. For a finite $\nu\in S$ over $\frak{N}$, let us choose a lift of $\overline{\rho}_{f,\frak{P}}|_{G_{L_\nu}}$ defining a point on the local ring $R_{\mathcal{O},\nu}^{\square,\mu}$. Let us now check the conditions of Thm. \ref{geetool} for our choice:
\begin{itemize}
\item Condition 1 holds as $\rho_{f,\frak{P}}$ is potentially diagonalizable, since $p>k+1$, so that the Hodge-Tate weights of $\rho_{f,\frak{P}}$ are in the Fontaine-Lafaille range. Hence, by base change its restriction to $G_L$ is also automorphic  (\cite{langlands}, assertion (A) of page 19) and also potentially diagonalizable, as this condition is obviously preserved for restriction base change. 
\item Condition 2 is also satisfied since $SL_2(\mathbb{F}_p)\subseteq\mathrm{Im}(\bar{\rho}_{f,\frak{P}})$, a condition preserved by restriction to $G_{L(\zeta_p)}$, since $p>\max\{k+1,6\}$ and $L/F$ is quadratic imaginary.  
\end{itemize}
Hence, we obtain an automorphic lift of $\bar{\rho}_{f,\frak{P}}|_{G_L}$ corresponding to a point in the deformation ring $\mathcal{R}_{global, L}$. This lift corresponds to an automorphic form $\pi$ of $GL_2(L)$ which is $\frak{p}$-supercuspidal for each prime $\frak{p}$ in $L$ over $p$, and potentially crystalline but it corresponds to a representation of $G_L$. Moreover, item (7) of Thm. \ref{geetool} grants that $\mathcal{R}_{global, L}$ is a finitely generated $\mathcal{O}$-module. Observe, furthermore, that since inertial types are constant on the irreducible components of these local deformation rings, due to items  (5) and (7), the modular representation produced by Thm. \ref{geetool} has the local inertial type at all primes in $S$ that we have specified.

We also need the following result, a proof of which can be adapted, mutatis mutandis, from the proof of \cite[Prop. 3.18]{blancodieulefait}:
\begin{prop} If  $\mathcal{R}_{global,L}$ is a finitely generated $\mathcal{O}$-module, then $\mathcal{R}_{global}$ is a finitely generated $\mathcal{O}$-module.
\label{fgeneration}
\end{prop}

Finally, we can prove our main result:

\begin{thm}There exists a Hilbert modular lift $\rho_{g, \frak{P'}}$ of $\bar{\rho}_{f,\frak{P}}$ where $g\in S_\textbf{2}(\frak{N}p^2)$ is supercuspidal at $\frak{p}$ for each prime $\frak{p}$ of $F$ over $p$.
\end{thm}
\begin{proof} We apply Thm. \ref{geetool} to $\bar{\rho}_{f,\frak{P}}|_{G_L}$ and we see from item (7) that $\mathcal{R}_{global,L}$ is a finitely generated $\mathcal{O}$-module and hence, by Prop. \ref{fgeneration} it follows that $\mathcal{R}_{global}$ is so.

From \cite[Lem. 2]{bockle} and Thm. \ref{cotadimension} it follows that $\mathcal{R}_{global}$ is a finite flat and complete intersections $\mathcal{O}$-module. Hence it follows that it contains a point $\mathcal{R}_{global}\to\bar{\Z}_p\to 0$. Since the restriction to $G_L$ of this lift is modular due to \cite[Thm. 4.2.1]{gee}, the lift itself is modular by solvable base change (see \cite[Thm. 6.2]{clozel} with $n=2$) and this finished the proof.
\end{proof}

Starting with a classical $h\in S_k(\Gamma_0(N))$ without CM, a prime $p>\max\{k+1,6\}$ such that $p\nmid N$, a prime $\frak{P}$ of $\mathcal{O}_h$ above $p$ such that the residual representation $\overline{\rho}_{h,\frak{P}}$ has large image, as a particular case of the previous result we have Thm. \ref{thm2}.

\section{Corrigenda to ``automorphy of $Sym^5(GL(2))$ and base change''}
In the paper \cite{dieulefait} there is a mistake in the process of introducing the micro good dihedral prime (MGD prime in what follows) $p=43$, which was noticed by Ariel Pacetti. The erroneous claim is the following: when a newform $f$ is principal series locally at a prime $p$, with ramification given by a character of conductor $p$ and prime order $t$, and we consider a residual mod $t$ Galois representation attached to $f$, the claim that this residual representation is either unramified or has unipotent ramification at $p$ is correct, but the claim that in the first case it is always possible to construct a lift with Steinberg ramification at $p$ is wrong (the mistake was caused by a miscalculation involving the unramified part of the characters describing the local at $p$ behaviour of the Galois representation at the prime $p$ in the case of a representation which is assumed to be in the principal series case). \\

To correct this mistake, we have to device an alternative way of introducing the MGD prime $43$ to the level. We note that the erroneous congruence is used in the paper not only in Step 4 of the safe chain, which is the step in which the MGD prime is introduced to the level, but also in Steps 3 and 7. Fortunately, Step 3 will now be fully replaced by another preparatory step, and the use of the wrong congruence in Step 7 is not essential, we will explain how it can be removed with just a minor change in this step. So what remains to be done is the following: \\

\begin{itemize}
\item[a)] Correct Step 7,
\item[b)] replace Step 3 by Step 3NEW, which is a preliminary step to the next one,
\item[c)] replace Step 4 by Step 4NEW, which explains the new method of introducing the MGD prime $43$ to the level. As we will see, this new method is precisely an application of the main theorem of this paper, i.e., we will use supercuspidal level raising mod $p$ at $p$.
\end{itemize}

 a) In Step 7 (see \cite{dieulefait}. Section 3.7) the wrong congruence is used only once, when playing with the pair of Sophie Germain primes $23$ and $47$, more precisely, when reducing mod $23$ a representation with ramification at $47$ given by a character of order $23$. If the residual mod $23$ representation ramifies at $47$ it must have unipotent ramification and the proof proceeds as it is. If it is unramified at $47$, we can argue as follows. First, we apply Ribet's level lowering to remove the prime $47$ from the level, obtaining a lift of the residual representation corresponding to a newform of level $43^2$ and weight $2$. Then, we apply a standard weight changing argument (multiplication by the Hasse invariant) to obtain a congruence mod $47$ between this modular form and a newform of level $43^2$ and weight $48$. From here the proof proceeds as in the other case, i.e., all arguments that follow apply to both cases. The only thing that requires extra explanation is the following: for the applications to the main results of Langlands functoriality obtained in \cite{dieulefait}, we need to work in each congruence that appears in the chain with representations that are either both ordinary or both potentially diagonalizable, so we need to explain why this is so in the new congruences that we have introduced. This is known for the case of the level lowering in characteristic $23$ because both representations are Barsotti-Tate, therefore potentially diagonalizable, so let us focus on the other congruence, the one where we have worked modulo $47$ and we have changed the weight. If the residual mod $47$ representation locally at $47$ is described by powers of the cyclotomic character, which corresponds to case (a) of Theorem \ref{arranque} with $r=1$, then we are in an ordinary case and the weight changing congruence amounts to moving in a Hida family, and we see that both representations are ordinary and crystalline in this case, thus both are potentially diagonalizable. \\
 In the complementary case, which corresponds to case (b) of Theorem \ref{arranque} with $k=2$, we can also see that both representations in the weight changing congruence are potentially diagonalizable: one of them is so because it is Barsotti-Tate (it is attached to a weight $2$ newform and the prime $47$ is not in its level), and for the other one we rely on Prop. 3.13 in \cite{ki1} where it is proved that the universal ring of local deformations in the case of crystalline deformations of weight $p+1 = 48$ is a domain, thus it is irreducible, plus the fact that can easily be checked (and is a particular case of Lemma 4.1.19 of \cite{blgg}) that in this case (case (b) of Theorem \ref{arranque} with $k=2$) we can construct a local lift of Hodge-Tate weights $0$ and $47$ which is induced from a crystalline character of the unramified quadratic extension of $\mathbb{Q}_p$. Thus, since the ring is irreducible and it contains a point given by the induction of a crystalline character, we deduce that all the points in this ring are potentially diagonalizable, and in particular the representation attached to the newform of weight $48$ that we consider in our congruence is potentially diagonalizable.

 Finally, let us mention that for the applications in \cite{dieulefait} it is also required that the residual image in each congruence is sufficiently large, so in principle we should check this in the new congruences introduced. However, this is checked in the paper in the first case (the case where the reduction mod $23$ of the newform has unipotent ramification at $47$), both for the congruence in characteristic $23$ and for the congruence in characteristic $47$, and it is easy to see that the same arguments work and give the same conclusion in the new cases we are considering. 

 b) Step 3NEW: This step replaces Step 3 in the paper (see \cite{dieulefait}, Section 3.3). Thus it begins with a newform of level $q^2$, good-dihedral at $q$, and with even weight $2 < k \leq 14$. Recall that, as explained in the paper, the good-dihedral prime $q$ ensures that, as long as we work exclusively in small residual characteristics (smaller than a prefixed bound), all residual images are going to be sufficiently large as required for the applications to the two main results of the paper. This step will be preparatory for step 4NEW, and what we need to do (the reasons for this will become clear in the next step) is to reduce to the case of weight $16$. Luckily, this is exactly what is done in step 7, also with an input newform having its weight in the  same range. So what we do in step 3NEW is exactly what is done in step 7 (of course, incorporating in step 7 the corrections explained above): we perform exactly the same set of congruences and we end up with a form of level $q^2$, good-dihedral at $q$, and of weight $16$. Notice that the only difference between step 3NEW and step 7 is that the input and output forms considered have different levels, the level being $q^2$ in the case of step 3NEW and $43^2$ in the case of step 7, but this does not change any of the arguments, it only simplifies the determination of the residual images because, as we have already remarked, having a good-dihedral prime $q$ in the level ensures sufficiently large residual images automatically (since in this step all congruences are in small characteristics).

 c) Step 4NEW: This replaces step 4 in the paper (see 
 \cite{dieulefait}, Section 3.4, we remark that the output of this step 
 will be the same as it was in step 4, so the chain can continue in step 
 5 in the paper). We will apply Thm. \ref{thm2} to introduce the MGD 
 prime $43$ to the level using a single congruence modulo $43$, doing 
 supercuspidal level raising. We start this step with the output of step
 3NEW, i.e., with a newform of level $q^2$, good-dihedral at $q$, and
 weight $16$. We consider a residual representation
 in characteristic $43$ attached to this newform. Observe that we are taking $p=43$ and we have $k=16$, $N= q^2$ thus the conditions $p > \max\{k+1, 6 \}$, $p \nmid N$ are satisfied. As explained in \cite{dieulefait}, the good-dihedral 
 prime $q$ in the level ensures that the image of this residual 
 representation is large enough as required  for the use of this congruence as part of the safe chain of congruences (i.e, for the application to the main theorems in \cite{dieulefait}). This also implies that the residual image satisfies the condition required to apply Thm. \ref{thm2}, so we apply this result and obtain a modular weight $2$ lift of this residual representation corresponding to a newform of level $43^2 \cdot q^2$ supercuspidal at $43$. More specifically, it follows from Proposition \ref{firstliftclassical} that the $43$-adic Galois representation just constructed is of type $(2, \tau)$ with $\tau \cong \omega_2^m \oplus \omega_2^{43 \cdot m} $ with $m = k + (p+1)(p-2) = 16 + 44 \cdot 41 = 1820$. Observe that the character $\omega_2^m$ has order $66$. \\ 
 At this point we will twist the modular form that we have obtained by a character of the form $\omega^a$ in order to reduce to a case where the character corresponding to the supercuspidal local type has order $11$. As explained in \cite{dieulefait}, twisting is also a valid move in the construction of the safe chain of congruences, and it can be freely applied whenever needed (this is due to the fact that it preserves automorphy and also the local properties that are required to apply Automorphy Lifting Theorems). It is an exercise in elementary arithmetic to check that by twisting by $\omega^a$ with $a = 35$ (or any other exponent congruent to $35$ modulo $42$) we obtain another weight $2$ newform supercuspidal at $43$ whose local type $\omega_2^{m + 44 \cdot a}
 \oplus \omega_2^{43 \cdot (m + 44 \cdot a)}$ corresponds to a character of order $11$. This finishes the procedure of introducing the MGD prime $43$ to the level, having obtained a weight $2$ newform with supercuspidal ramification at $43$ given by a character of order $11$.
 For the application to the main results in \cite{dieulefait}, it is important to stress that in the congruence that we have performed in this step both representations are potentially diagonalizable, one of them because it is in the Fontaine-Laffaille case, and the other one because it is a potentially Barsotti-Tate representation.

\end{document}